\documentclass{artjlt}

\usepackage{amssymb,amsmath}
\usepackage{amsthm}
\usepackage{bbm}
\usepackage{stmaryrd}
\usepackage{hyperref}
\usepackage[usenames,dvipsnames]{xcolor}
\usepackage{color}

\usepackage[all,cmtip]{xy}

\numberwithin{equation}{section}

\newcommand{\Id}{\mathbbmss{1}}

\newcommand{\rmd}{\textnormal{d}}

\newcommand{\rmh}{\textnormal{h}}

\newcommand{\rmS}{\textnormal{S}}

\DeclareMathOperator{\End}{End}
\DeclareMathOperator{\rank}{rk}

\font\black=cmbx10 \font\sblack=cmbx7 \font\ssblack=cmbx5 \font\blackital=cmmib10  \skewchar\blackital='177
\font\sblackital=cmmib7 \skewchar\sblackital='177 \font\ssblackital=cmmib5 \skewchar\ssblackital='177
\font\sanss=cmss12 \font\ssanss=cmss8 
\font\sssanss=cmss8 scaled 600 \font\blackboard=msbm10 \font\sblackboard=msbm7 \font\ssblackboard=msbm5
\font\caligr=eusm10 \font\scaligr=eusm7 \font\sscaligr=eusm5  \font\fraktur=eufm10
\font\sfraktur=eufm7 \font\ssfraktur=eufm5 
\font\bsymb=cmsy10 scaled\magstep2
\def\all#1{\setbox0=\hbox{\lower1.5pt\hbox{\bsymb
       \char"38}}\setbox1=\hbox{$_{#1}$} \box0\lower2pt\box1\;}
\def\exi#1{\setbox0=\hbox{\lower1.5pt\hbox{\bsymb \char"39}}
       \setbox1=\hbox{$_{#1}$} \box0\lower2pt\box1\;}

\def\tx#1{{\fam0\relax#1}}

\newfam\bifam
\textfont\bifam=\blackital \scriptfont\bifam=\sblackital \scriptscriptfont\bifam=\ssblackital

\newfam\blfam
\textfont\blfam=\black \scriptfont\blfam=\sblack \scriptscriptfont\blfam=\ssblack

\newfam\bbfam
\textfont\bbfam=\blackboard \scriptfont\bbfam=\sblackboard \scriptscriptfont\bbfam=\ssblackboard

\newfam\ssfam
\textfont\ssfam=\sanss \scriptfont\ssfam=\ssanss \scriptscriptfont\ssfam=\sssanss
\def\sss#1{{\fam\ssfam\relax#1}}

\newfam\clfam
\textfont\clfam=\caligr \scriptfont\clfam=\scaligr \scriptscriptfont\clfam=\sscaligr

\newfam\frfam
\textfont\frfam=\fraktur \scriptfont\frfam=\sfraktur \scriptscriptfont\frfam=\ssfraktur

\def\hpb#1{\setbox0=\hbox{${#1}$}
    \copy0 \kern-\wd0 \kern.2pt \box0}
\def\vpb#1{\setbox0=\hbox{${#1}$}
    \copy0 \kern-\wd0 \raise.08pt \box0}

\def\pmb#1{\setbox0\hbox{${#1}$} \copy0 \kern-\wd0 \kern.2pt \box0}
\def\pmbb#1{\setbox0\hbox{${#1}$} \copy0 \kern-\wd0
      \kern.2pt \copy0 \kern-\wd0 \kern.2pt \box0}
\def\pmbbb#1{\setbox0\hbox{${#1}$} \copy0 \kern-\wd0
      \kern.2pt \copy0 \kern-\wd0 \kern.2pt
    \copy0 \kern-\wd0 \kern.2pt \box0}
\def\pmxb#1{\setbox0\hbox{${#1}$} \copy0 \kern-\wd0
      \kern.2pt \copy0 \kern-\wd0 \kern.2pt
      \copy0 \kern-\wd0 \kern.2pt \copy0 \kern-\wd0 \kern.2pt \box0}
\def\pmxbb#1{\setbox0\hbox{${#1}$} \copy0 \kern-\wd0 \kern.2pt
      \copy0 \kern-\wd0 \kern.2pt
      \copy0 \kern-\wd0 \kern.2pt \copy0 \kern-\wd0 \kern.2pt
      \copy0 \kern-\wd0 \kern.2pt \box0}

\mathchardef\za="710B  
\mathchardef\zb="710C  
\mathchardef\zg="710D  
\mathchardef\zd="710E  
\mathchardef\zve="710F 
\mathchardef\zz="7110  
\mathchardef\zh="7111  
\mathchardef\zvy="7112 
\mathchardef\zi="7113  
\mathchardef\zk="7114  
\mathchardef\zl="7115  
\mathchardef\zm="7116  
\mathchardef\zn="7117  
\mathchardef\zx="7118  
\mathchardef\zp="7119  
\mathchardef\zr="711A  
\mathchardef\zs="711B  
\mathchardef\zt="711C  
\mathchardef\zu="711D  
\mathchardef\zvf="711E 
\mathchardef\zq="711F  
\mathchardef\zc="7120  
\mathchardef\zw="7121  
\mathchardef\ze="7122  
\mathchardef\zy="7123  
\mathchardef\zf="7124  
\mathchardef\zvr="7125 
\mathchardef\zvs="7126 
\mathchardef\zf="7127  
\mathchardef\zG="7000  
\mathchardef\zD="7001  
\mathchardef\zY="7002  
\mathchardef\zL="7003  
\mathchardef\zX="7004  
\mathchardef\zP="7005  
\mathchardef\zS="7006  
\mathchardef\zU="7007  
\mathchardef\zF="7008  
\mathchardef\zW="700A  
\mathchardef\zC="7009  

\newcommand{\be}{\begin{equation}}
\newcommand{\ee}{\end{equation}}

\newcommand{\bea}{\begin{eqnarray}}
\newcommand{\eea}{\end{eqnarray}}
\def\*{{\textstyle *}}
\newcommand{\R}{{\mathbb R}}

\newcommand{\N}{{\mathbb N}}

\newcommand{\s}{{\textstyle *}}

\newcommand{\A}{{\cal A}}

\def\Sec{\sss{Sec}}
\def\Vect{\sss{Vect}}

\def\sT{{\sss T}}

\def\sv{{\sss v}}

\def\xi{\tx{i}}

\def\s*{{\scriptstyle *}}

\def\cE{\mathcal{E}}

\title{Representations up to homotopy from weighted Lie algebroids}                                     
\author{Andrew James Bruce, Janusz Grabowski and Luca Vitagliano}                 
\lastname{Bruce, Grabowski and Vitagliano}  

\msc{16W50, 22A22, 53D17}    

\keywords{graded manifolds, Lie algebroids, Lie groupoids, representations up to homotopy}         

\address{%
Andrew James Bruce\\            
Mathematics Research Unit\\
University of Luxembourg \\
Maison du Nombre 6\\
avenue de la Fonte\\
L-4364, Esch-sur-Alzette,  Luxembourg\\
andrewjamesbruce@googlemail.com}
\address{%
Janusz Grabowski\\ 
Institute of Mathematics\\
Polish Academy of Sciences\\
\'Sniadeckich 8 \\
00-656 Warszawa, Poland
\\ jagrab@impan.pl}
\address{%
Luca Vitagliano\\ 
Department of Mathematics\\
Universit\`{a} degli Studi di Salerno\\
Via Giovanni Paolo II n. 123\\
84084 Fisciano (SA), Italy\\
 lvitagliano@unisa.it}

\begin{document}

\title{Representations up to homotopy\\ from weighted Lie algebroids}

\maketitle

\begin{abstract}
Weighted Lie algebroids were recently introduced as Lie algebroids equipped with an additional compatible non-negative grading, and represent a wide generalisation of the notion of a $\mathcal{VB}$-algebroid. There is a close relation between two term representations up to homotopy of Lie algebroids and $\mathcal{VB}$-algebroids. In this paper we show how this relation generalises to weighted Lie algebroids and in doing so we uncover new and natural examples of higher term representations up to homotopy of Lie algebroids. Moreover, we show how the van Est theorem generalises to weighted objects.
\end{abstract}

\section{Introduction}\label{sec:Intro}
Lie groupoids are an integral part of  differential geometry. Generically, these objects encode quite a general notion of a symmetry and as such can be found in various guises throughout modern  geometry.  Lie algebroids are the infinitesimal counterpart of Lie groupoids, and include as extreme examples Lie algebras and tangent bundles. Our general reference for the theory of Lie groupoids and Lie algebroids is the book by Mackenzie \cite{Mackenzie2005}.\par
It is known, thanks to the work of  Gracia-Saz \& Mehta \cite{Gracia-Saz:2009},  that $\mathcal{VB}$-algebroids, i.e., vector bundle objects in the category of Lie algebroids, are the natural geometric setting  for two term representations up to homotopy (cf. Arias Abad \& Crainic \cite{Abad:2012}).
 Weighted Lie algebroids \cite{Bruce:2014b,Bruce:2015,Bruce:2016} are a far reaching generalisation of $\mathcal{VB}$-algebroids and so it is natural to wonder what the link between  weighted Lie algebroids  and representations  up to homotopy of Lie algebroids is. In this paper we clarify these links and in doing so open up many natural examples of representations up to homotopy of Lie algebroids. \par
We will show that associated with any weighted Lie algebroid is a series of canonical Lie algebroid modules (in the sense of Va\u{\i}ntrob \cite{Vaintrob:1997})  over the underlying weight zero Lie algebroid. Moreover, to construct these Lie algebroid modules one does not need a splitting of the weighted Lie algebroid -- we really do have a canonical construction. Using the results of Mehta \cite{Mehta:2014}, we know that a Lie algebroid module is (up to  non-canonical choice) equivalent to a representation up to homotopy of the Lie algebroid. From this perspective, we find that the notion of a Lie algebroid module is the most natural definition of a generalised representation of a Lie algebroid.\par
Moving on to the global counterparts, weighted Lie groupoids were first defined and studied in \cite{Bruce:2015} and offer a wide generalisation of the notion of a $\mathcal{VB}$-groupoid, i.e.~a vector bundle object in the category of Lie groupoids . There has been some renewed interest in $\mathcal{VB}$-groupoids. In particular, Cabrera \& Drummond \cite{Cabrera:2016} proved a version of the van Est theorem (see Crainic \cite{Crainic:2003}) for homogeneous cochains of a $\mathcal{VB}$-groupoid. As we shall show, a refined version of the van Est theorem holds for weighted Lie groupoids, and in fact follows from minor adjustments to the ideas and proofs presented by  Cabrera \& Drummond.

\smallskip

\noindent \textbf{Arrangement:}   In Section \ref{sec:preliminaries} we recall the notions of a graded bundle, Lie algebroid modules and representations up to homotopy. The informed reader may safely skip these preliminaries.  In Section \ref{sec:WeightedAlgebroids} we present the bulk of this paper and explain how to associate in a canonical way a series of Lie algebroid modules to a given weighted Lie algebroid. We conclude this paper in  Section \ref{sec:VanEst} where we present a version of the  van Est theorem  for weighted Lie groupoids.

\section{Preliminaries}\label{sec:preliminaries}
\subsection{Graded bundles and homogeneity structures}
In this paper we will focus on a particular class of graded manifolds known as \emph{graded bundles} (see \cite{Grabowski:2013,Grabowski:2012}). There is a more general notion of \emph{graded (super)manifolds}  in the spirit of Th.~Voronov \cite{Voronov:2001}  (see also \cite{Roytenberg:2001} and \cite{Severa:2005}). However,  such supermanifolds will not be employed in the main body of this paper. Canonical examples of graded bundles  include vector bundles and higher tangent bundles. The grading on the structure sheaf of a graded bundle is equivalent to a smooth action of the multiplicative monoid of reals
$$\rmh : \mathbb{R} \times F \rightarrow F,$$
which is known as a \emph{homogeneity structure}. Recall that in this context a function $f \in C^{\infty}(F)$ is homogeneous of weight $w$ if and only if
$$\rmh_t^*f = t^wf,$$
for all $t \in \R$. We call $w$  the \emph{weight} of  $f$. Grabowski \& Rotkiewicz \cite{Grabowski:2012}  proved that only \emph{non-negative integer weights} are allowed, so the algebra $\mathcal{A}(F)\subset C^\infty(F)$ spanned by homogeneous functions has a canonical $\mathbb{N}$-grading, $\mathcal{A}(F) = \bigoplus_{i \in \mathbb{N}}\mathcal{A}^{i}(F)$. This algebra  is referred to as the \emph{algebra of polynomial functions} on $F$. \par
Importantly, we have  that for $t \neq 0$ the action $\rmh_{t}$ is a diffeomorphism of $F$ and when $t=0$ it is a surjective submersion $\tau=\rmh_0$ onto a submanifold $F_{0}=M$, with the fibres being diffeomorphic to $\mathbb{R}^{N}$ for some $N$  (cf.  \cite{Grabowski:2012}). Moreover, the objects obtained are particular kinds of \emph{polynomial bundles} $\tau:F\to M$, i.e. fibrations which locally look like $U\times\R^N$ and the change of coordinates (for a certain choice of an atlas) are polynomial in $\R^N$. For the dual picture see \cite{Grabowski:2016}. A homogeneity structure is said to be \emph{regular} if and only if
\begin{equation*}
\left.\frac{\rmd }{\rmd t}\right|_{t=0}\rmh_{t}(p) = 0  \hspace{20pt} \Longrightarrow \hspace{20pt} p = \rmh_{0}(p),
\end{equation*}
 for all points $p \in F$ (see \cite{Grabowski:2009}). Moreover, if the homogeneity structure is regular then we have a vector bundle structure,  i.e.~$\rmh$ is the scalar multiplication in the fibres of a vector bundle $F \to F_0$ . The converse is also true, i.e.~vector bundles are graded bundles whose homogeneity structure is regular.\par
We can always employ homogeneous local coordinates of the form $(x^{a}, y^{\alpha}_{w})$, where $(x^{a})$ form a (weight zero)  coordinate system on $M := F_0$ and the fibre coordinates are labelled by their  (positive) weight $w \in \mathbb{N}$ \cite{Grabowski:2012}  (for the case of supermanifolds consult \cite{Jozwikowski:2016}).  Changes of coordinates must respect the weight and are polynomial in the non-zero weight coordinates. The weight of the highest weight coordinates we refer to as the \emph{degree} of the graded bundle.  For instance, vector bundles are exactly degree 1 graded bundles.  A graded bundle  of degree $k$  admits a sequence of  polynomial fibrations  $F \to F_i$, where a point of $F_i$ is a class of the points of $F$ described  in a  homogeneous  coordinate system by the coordinates of weight $\leq i$, with the obvious tower of  surjections
\begin{equation}\label{eqn:fibrations}
F=F_{k} \stackrel{\tau^{k}}{\longrightarrow} F_{k-1} \stackrel{\tau^{k-1}}{\longrightarrow}   \cdots \stackrel{\tau^{3}}{\longrightarrow} F_{2} \stackrel{\tau^{2}}{\longrightarrow}F_{1} \stackrel{\tau^{1}}{\longrightarrow} F_{0} = M,
\end{equation}
where the coordinates on $M$ have zero weight. Note that  $F_{1} \rightarrow M$ is a \emph{linear fibration}, and so we have a vector bundle structure,  and that the other fibrations $F_{l} \rightarrow F_{l-1}$ are \emph{affine fibrations} in the sense that the changes of local coordinates for the fibres are linear plus additional additive terms of appropriate weight.
\begin{proposition}[\cite{Bruce:2016}]\label{prop:split}
An arbitrary  graded bundle $F_k$  is non-canonically isomorphic to a split graded bundle, i.e., a Whitney sum of vector bundles  with shifted degrees in the fibers.
\end{proposition}
To be a little more specific,  the model vector bundle $\sv(F_i)$ of the affine bundle $\tau^i : F_i \rightarrow F_{i-1}$ turns out to be the pull-back of a vector bundle $\bar\sv(F_i) \to M$ over $M$ (cf. \cite{Bruce:2018}) and
we have an isomorphism
$$F_k \cong \bar\sv(F_k) \oplus_M \bar\sv(F_{k-1}) \oplus_M \dots \oplus_M \bar\sv(F_1) \,,$$
We stress that the isomorphism here is far from unique and nearly never canonical. \par
\begin{example} The principal canonical example of a graded bundle is the \emph{higher tangent bundle}; i.e. the manifold of $k$-th jets (at zero) of curves $\gamma: \mathbb{R} \rightarrow M$ which we denote by $\sT^{k}M$.  Furthermore, it is well known that  $\sT^{k}M$ is non-canonically isomorphic to $\underbrace{\sT M \oplus_M \cdots \oplus_M \sT M}_{k \textnormal{ times}}$ via a choice of an affine connection on $M$. 
\end{example}
\begin{example}
The previous example should \emph{not} be confused with the \emph{iterated tangent bundle} which we denote by $\sT^{(k)}M$ and is defined via application of the tangent functor $k$-times. $\sT^{(k)}M$ is naturally a multi-graded bundle where each homogeneity structure is associated with vector bundle structures $\sT^{(l)}M \rightarrow \sT^{(l-1)}M$, for $1 \leq l \leq k$. Here we define $\sT^{(0)}M := M$. By passing to total weight, i.e., summing the individual components of the multi-weight, one obtains a graded bundle of degree $k$.
\end{example}

\subsection{Recollection of Lie algebroids}
Recall the notion of a Lie algebroid as a vector bundle $\pi:A \rightarrow M$ equipped with a Lie bracket on the sections $[-, -]: \Sec(A) \times \Sec(A) \rightarrow \Sec(A)$ together with an anchor,  i.e.~a vector bundle morphism $\rho : A \to \sT M$, inducing a $C^\infty (M)$-linear map $\Sec (A) \to \Vect (M)$, also denoted by $\rho$ , that satisfy the Leibniz rule
\begin{equation*}
 [u,fv] = \rho(u)f \: v  +   f [u,v],
\end{equation*}
\noindent for all $u,v \in \Sec(A)$ and $f \in C^{\infty}(M)$. The Leibniz rule implies that the anchor is actually a Lie algebra morphism: $\rho\left([u,v]\right) = [\rho(u), \rho(v)]$.  A Lie algebroid is a triple $(A, [-,-], \rho)$, however we will typically denote a Lie algebroid simply as $A$ when no confusion can arise.\par
If we pick some local basis for the sections $\{s_i\}$, then the structure functions of a Lie algebroid are defined by
\begin{equation}\label{Q}
[s_{i} , s_{j}] = Q_{ij}^{k}(x)s_{k}\,,\quad
\rho(s_{i})= Q_{i}^{a}(x)\frac{\partial}{\partial x^{a}}\,,
\end{equation}
\noindent and satisfy the  \emph{Lie algebroid structure equations}
\begin{align*}
&Q_{ij}^{k} + Q_{jk}^{i} =0,&
& Q_{i}^{a}\frac{\partial Q_{j}^{b}}{\partial x^{a}} - Q_{j}^{a}\frac{\partial Q_{i}^{b}}{\partial x^{a}}=0,&
& \sum_{\textnormal{cyclic}}\left( Q_{i}^{a}\frac{\partial Q_{jk}^{l}}{\partial x^{a}} - Q_{im}^{l}Q_{jk}^{m} \right)=0.&
\end{align*}
The two extreme examples of Lie algebroids are of course tangent bundles and Lie algebras. In a sense, general examples are a mixture of these two fundamental examples and include integrable distributions, Lie algebra bundles and the cotangent bundle of a Poisson manifold. The reader may consult Mackenzie \cite{Mackenzie2005} for further examples.\par
Given any Lie algebroid $A$, there is a differential $\rmd_A$ defined on the graded algebra of  \emph{Lie algebroid forms}, $ \Omega^{\bullet}(A) := \Sec\big(\bigwedge^\bullet A^* \big)$, often referred to as  $A$-forms  for brevity.  To be precise, given any Lie algebroid the  \emph{de Rham differential}
$$\rmd_A : \Omega^{k}(A) \longrightarrow \Omega^{k+1}(A)$$
is defined as
\begin{eqnarray*}
 \rmd_A\omega(X_0, \dots ,X_k ) &=& \sum_{i=0}^k (-1)^i \rho(X_i) \: \omega(X_0, \dots, \widehat{X_i} , \dots , X_k)\\
  & & + \sum_{i<j} (-1)^{i+j}\omega([X_i , X_j], X_0, \dots , \widehat{X_i} , \dots, \widehat{X_j}, \dots X_k),
\end{eqnarray*}
where $\omega \in \Omega^k(A)$ and $X_0, \dots , X_k \in \Sec(A)$. The Lie algebroid structure equations can be shown to be directly equivalent to $\rmd_A^2=0$, and so we do obtain a differential in this way. For tangent bundles, the resulting differential is the standard de Rham differential and for Lie algebras we have the Chevalley--Eilenberg differential.
\begin{Definition}
The \emph{Chevalley--Eilenberg--de Rham cochain complex} associated  with a Lie algebroid $(A, [-,-], \rho)$ is the  cochain complex of Lie algebroid forms $\big ( \Omega^{\bullet}(A), \: \rmd_{A} \big)$. The resulting cohomology is known as the \emph{standard Lie algebroid cohomology}.
\end{Definition}
In adapted local coordinates $(x^a, y^i)$ on $A$, Lie algebroid forms can be written as
$$\omega = \sum_{n=0}^q \frac{1}{n!} y^{i_1}\wedge \dots \wedge y^{{i_n}}\omega_{i_n \dots i_1}(x).$$
 Here we interpret a fiber-wise linear function on $A$, e.g.~linear fibre coordinates $y^i$ as degree $1$ A-forms. We will adopt the same interpretation for every vector bundle in the sequel without further comments.  We can describe the differential $\rmd_A$ locally  using
\begin{align*}
& \rmd_A x^a  = y^iQ_i^a(x), && \rmd_A y^k = \frac{1}{2}y^i \wedge y^j Q_{ji}^k(x),
\end{align*}
together with the Leibniz rule.\par
Suppose we are given two Lie algebroids $A_1$ and $A_2$. Then a \emph{morphism} from $A_1$ to $A_2$ is a vector bundle morphism $\phi : A_1 \rightarrow A_2$, such that 
$\phi^*\circ\rmd_{A_2}=\rmd_{A_1}\circ\phi^*$, where $\phi^*:\Omega^\bullet(A_2)\rightarrow \Omega^\bullet(A_1)$
is the map induced by $\phi$ on `forms'.

\subsection{Lie algebroid modules and representations up to homotopy}
Our general understanding of Lie algebroid modules comes from  Va\u{\i}ntrob \cite{Vaintrob:1997} who formulated them in terms of supermanifolds. Here we adapt the definition to a more classical picture in terms of homological algebra. In our description the correspondence between Lie algebroid modules and representations up to homotopy will be obvious. Before we go on, we note that the  graded vector bundles $\mathcal{V} = \oplus_i \mathcal{V}_i$ used in the following definitions are generally $\mathbb Z$-graded and \emph{not} just $\mathbb N$-graded. For simplicity, we will insist that the graded vector bundles are \emph{degree bounded}, though this restriction need not be enforced in general. To explain this, we denote the rank of $\Sec(\mathcal{V})$ in degree $i$ as  $\rank_i(\mathcal{V})$. Then we say that $\mathcal{V}$ is degree bounded if there exist integers $m,n$ such that $\rank_i(\mathcal{V}) =0$ if $i<m$ or $i>n$.
\begin{Definition}\label{def:LieModule}
Let $\big(A, [-,-], \rho  \big)$ be a Lie algebroid and let $\big(  \Omega^\bullet(A), \rmd_A\big)$ be the associated Chevalley--Eilenberg--de Rham cochain complex. Then a \emph{Lie algebroid module} over $A$ (an \emph{$A$-module} for brevity) is a differential graded $(\Omega^\bullet (A), \rmd_A)$-module, i.e.~a cochain complex $\big(V, \rmd_V \big)$ where
\begin{enumerate}
\item $V$ is a graded $\Omega^\bullet (A)$-module;
\item $\rmd_V$ satisfies the Leibniz rule
$$\rmd_V(\omega \cdot v) = \rmd_A \omega \cdot v + (-1)^{\deg(\omega)} \omega \cdot \rmd_V v,$$
for all (homogeneous) $\omega \in \Omega^\bullet(A)$ and $v \in V$

\item as a $\Omega^\bullet (A)$-module, $V$ is isomorphic to $\Omega^\bullet(A)\otimes_{C^\infty(M)} \Sec(\mathcal{V})$, where $\mathcal{V}$ is a graded vector bundle over $M$. Note that we do not insist that this isomorphism be fixed.
\end{enumerate}
\end{Definition}
\begin{Definition}
 Let $(V_1, \rmd_{V_1}) $ and  $(V_2, \rmd_{V_2}) $ be Lie algebroid modules over the same Lie algebroid $A$. Then a \emph{morphism of $A$-modules} is an $\Omega^\bullet (A)$-linear cochain map $\varphi : V_1 \rightarrow V_2 $, i.e., $\varphi (\omega \cdot v) = \omega \cdot \varphi (v)$, and $\varphi\circ \rmd_{V_1} = \rmd_{V_2} \circ \varphi$.
\end{Definition}
The notion of an automorphisms of a Lie algebroid module and the resulting isomorphism classes is clear.
Recall that a representation up to homotopy of a Lie algebroid $A \rightarrow M$ is a cochain complex $(\mathcal{V}, \partial)$ of vector bundles over $M$ equipped with an $A$-connection and maps $\omega_{i} : \bigwedge^{i}\Sec(A) \rightarrow \End^{1-i}(\mathcal{V})$ for $i \geq 2$ that satisfy some series of coherence conditions. The first of these conditions stated that $\omega_{2}$ generates the chain homotopies controlling the curvature of the $A$-connection $\omega_1$. Very loosely, the connection is `flat up to homotopy'. From the start of the subject, it was realised that superconnections and Lie algebroid modules provide an economical and precise way to describe representations up to homotopy of Lie algebroids. The definition we take, as originally given by Arias Abad \& Crainic \cite{Abad:2012} is:
\begin{Definition}\label{def:RepHomot}
Let $\big(A, [-,-], \rho \big)$ be a Lie algebroid and let $\mathcal{V} = \bigoplus \mathcal{V}_{i}$ be a graded vector bundle over $M$. Furthermore, let us denote the space of $\mathcal{V}$-valued $A$-forms as
$$\Omega(A, \mathcal{V}) :=  \Omega^\bullet(A) \otimes_{C^{\infty}(M)}\Sec(\mathcal{V}).$$
Then a \emph{representation up to homotopy} (a.k.a.~an $\infty$-representation) of $A$ on $\mathcal{V}$ is a (total) degree $1$ operator $\mathcal{D}$ on $\Omega(A, \mathcal{V})$ that squares to zero, i.e., $\mathcal{D}^2 =0$, and satisfied the Leibniz rule
$$\mathcal{D}(\omega \:\mathsf{v}) = \rmd_{A}(\omega) \: \mathsf{v} + (-1)^{\deg(\omega)}\omega \: \mathcal{D}(\mathsf{v}),$$
for any and all $\omega \in \Omega^\bullet(A)$ and $\mathsf{v} \in \Omega(A, \mathcal{V})$  (in other words $(\Omega^\bullet (A, \mathcal V), \mathcal D)$ is a differential graded $(\Omega^\bullet (A), \rmd_A)$-module). The operator $\mathcal{D}$ will be referred to as a \emph{flat superconnection}.
\end{Definition}
We will need the notion of representations up to homotopy being \emph{gauge equivalent} as this sets up the notion of the relevant isomorphism classes. Note that we have a natural projection map
$$\varrho : \Omega(A, \mathcal{V}) \longrightarrow \Sec(\mathcal{V}),$$
the kernel of which is $\bigoplus_{p>0} \Omega^p(A)\otimes_{C^{\infty}(M)}\Sec(\mathcal{V})$. Using this projection we can define gauge transformation.
\begin{Definition}
A \emph{gauge transformation} of $\Omega(A, \mathcal{V})$ is a degree preserving $\Omega^\bullet(A)$-module automorphism $\varphi$ that acts as the identity on $\Sec(\mathcal{V})$, i.e., $\varrho \circ \varphi = \Id_\mathcal{V} \circ \varphi$. Under a gauge transformation the flat superconnection transforms as $\mathcal{D}^\varphi := \varphi^{-1} \circ \mathcal{D} \circ \varphi$.
 \end{Definition}
Two flat superconnections are said to be \emph{gauge equivalent} if they are related by  a gauge transformation. Naturally, isomorphism classes of representations up to homotopy are identified  by collections of gauge equivalent flat superconnections.\par

From the definitions it is almost obvious that Lie algebroid modules are essentially equivalent to representations up to homotopy. Once an isomorphism as $\Omega^\bullet (A)$-modules $V \rightarrow \Omega(A, \mathcal{V})$ as been chosen one can then equip $\Omega(A, \mathcal{V})$ with a flat superconnection.  Different choices of this isomorphism lead to gauge equivalent representations up to homotopy. As far as we know, this essential equivalence was first clearly stated by Mehta \cite[Theorem 4.5]{Mehta:2014}, albeit using supermanifolds.
\begin{theorem}[Mehta]\label{thm:Mehta}
Let $(A, [-,-], \rho)$ be a Lie algebroid. There is a one-to-one correspondence between isomorphism  classes  of $A$-modules and  isomorphism  classes  of  representations up to homotopy of $A$.
\end{theorem}
All the structures we will meet in this paper will be concentrated in non-negative degrees, or in other words $\rank_i(\mathcal{V}) =0$ when $i<0$. For example, a representation up to homotopy is said to be  an \emph{$m$-term representation up to homotopy} if the graded vector bundle $\mathcal{V}$ is concentrated in non-negative degrees and the degree is bounded by $m$.

\subsection{$\mathcal{VB}$-algebroids revisited}\label{subsec:VBalgebroids}
 Before we discuss the general setting of weighted Lie algebroids it will be illustrative to consider $\mathcal{VB}$-algebroids defined in terms of Lie algebroids equipped with a compatible regular homogeneity structures.
\begin{Definition}\label{def:VB-alg}
A $\mathcal{VB}$-algebroid is a Lie algebroid $\big(E,[-,-]_E, \rho_E  \big)$ equipped with a regular homogeneity structure
$$\rmh :  \R \times E \longrightarrow E,$$
that acts by Lie algebroid morphisms for all $t \in \R$.
\end{Definition}
\begin{remark}
Bursztyn, Cabrera and del Hoyo  \cite[Theorem~3.4.3]{Bursztyn:2014} establish  that the standard and much more complicated definition of a $\mathcal{VB}$-algebroid is equivalent to Definition \ref{def:VB-alg},  i.e.~to the existence of a compatible regular homogeneity structure on a standard Lie algebroid.
\end{remark}
From the definition of a $\mathcal{VB}$-algebroid we have a double vector bundle  in the sense of Pradines \cite{Pradines:1974}
\be\label{dvb}
\xymatrix{E \ar[r] \ar[d] & A \ar[d] \\
B \ar[r] & M}
\ee
such that the the vector bundles $ \tau : E \rightarrow B$ and $\pi: A \rightarrow M$ are Lie algebroids. The first structure is clear from the definition, the second Lie algebroid is defined by the map $\rmh_0 : E \rightarrow A$, which by definition must be a Lie algebroid morphism. We refer to the Lie algebroid $\big( A, [-,-] , \rho\big)$ defined by the projection $\rmh_0$ as the \emph{underlying degree zero Lie algebroid}.  We stress that here, and in what follows, $\rmh$ is the (regular) homogeneity structure corresponding to vector bundle $E \to A$. Nonetheless, $E$ possesses another (regular) homogeneity structure corresponding to vector bundle $E \to B$, and the axioms of the double vector bundle guarantee that $E$ can be coordinatised by coordinates which are homogeneous with respect to both homogeneity structures (see \cite[Theorem 3.2.]{Grabowski:2009}):

\begin{equation}\label{eq:hom_chart}
\big( \underbrace{x^a}_{(0,0)} ,~ \underbrace{y^i}_{(0,1)} , ~ \underbrace{z^\alpha}_{(1,0)},~ \underbrace{w^l}_{(1,1)}   \big).
\end{equation}
Here we have indicated the corresponding bi-weights, the first one being that associated with $\rmh$. So, $(x,y)$, $(x, z)$ and $(x, w)$ serve as coordinate systems on $A$, $B$ and $C :=  \ker (E \to A) \cap \ker (E \to B)$ the \emph{core} of the double vector bundle, respectively.  \par
As $E \to B$ is a Lie algebroid we have a Chevalley--Eilenberg--de Rham cochain complex $\big( \Omega^\bullet(E), \: \rmd_E  \big)$. The homogeneity structure $\rmh$ extends to an action of the multiplicative reals on $E$-forms via pull-back. This action is by cochain maps. In particular,  the subspace in $\Omega^\bullet (E)$ spanned by homogeneous cochains is a subcomplex and it is naturally bi-graded:  we have the standard cohomological degree of Lie algebroid forms, but also the degree associated with the homogeneity structure.  Similarly as above, we will fix the order of this bi-degree to be first the weight associated with $\rmh$ and then the cohomological degree. With this ordering the differential $\rmd_E$ is of bi-degree $(0,1)$.
Homogeneous $E$-forms of weight $1$ will be denoted by $\Omega^{(1, \bullet)}(E)$.  They are locally given by
\begin{align*}
&\Omega^{(1,0)}(E) \cong \Sec(B^*)  \ni z^\alpha \sigma_{\alpha}(x)\\
& \Omega^{(1,1)}(E)  \ni y^i z^\alpha \sigma_{\alpha i}(x)  + w^l \sigma_l(x)\\
& \Omega^{(1,2)}(E) \ni \frac{1}{2} y^i \wedge y^j z^\alpha \sigma_{\alpha ji}(x) + y^i \wedge w^l \sigma_{li}(x)\\
&    \hspace{60pt} \vdots\\
& \Omega^{(1,q)}(E) \ni \frac{1}{q!} y^{i_1} \wedge  \dots \wedge y^{i_q} z^\alpha \sigma_{\alpha i_q \dots i_1}(x) + \frac{1}{(q-1)!} y^{i_1} \wedge  \dots \wedge y^{i_{q-1}} \wedge w^l \sigma_{l i_{q-1} \dots i_1}(x) \\
&\Omega^{(1, q+1)}(E)  \ni \frac{1}{q!} y^{i_1} \wedge  \dots \wedge y^{i_q} \wedge w^l \sigma_{l i_q \dots i_1}(x)
\end{align*}
where $q = \rank(A)$. All higher degree forms of weight $1$ are identically zero. It  is clear that $\big( \Omega^{(1, \bullet)}(E), \rmd_E \big)$ is an $A$-module. This $A$-module is canonically associated with the $\mathcal{VB}$-algebroid structure on $E$.\par
Every double vector bundle is non-canonically isomorphic to a split double vector bundle, i.e., the Whitney sum of the two side vector bundles and the core vector bundle. Locally this was first established in \cite{Grabowski:2009} and then using a \v{C}ech cohomology argument one can extend this to a global isomorphism as suggested in \cite{Gracia-Saz:2009}. Alternatively, a minor modification of  Proposition \ref{prop:split}, and in particular the proof given in \cite{Bruce:2016}, yields the result.  Now suppose we have chosen such an isomorphism $E \simeq A \times_M B \times_M C$. Then associated with this choice there is an isomorphism of $C^\infty(M)$-modules
$$\Omega^{(1,\bullet)}(E) \overset{\varphi}{\longrightarrow} \Omega^\bullet(A, B^* \oplus C^* \big),$$
 where $B^\ast$ sits in degree $0$, and $C^\ast$ sits in degree $1$. We can then use the total degree to define a graded vector bundle $\mathcal{V} := B^* \oplus C^*$. Moreover, we can use the isomorphisms of modules as a gauge transformation and define a flat superconnection as
$$\mathcal{D} :=  \varphi \circ \rmd_E \circ \varphi^{-1}.$$
Thus, associated with any $\mathcal{VB}$-algebroid is, up to a non-canonical choice (i.e., a gauge transformation), a two term representation up to homotopy (see Definition \ref{def:RepHomot}). In the other direction, one can build a $\mathcal{VB}$-algebroid from a two term representation up to homotopy $\big(\Omega^\bullet(A, B^* \oplus C^*), \mathcal D \big)$, using $A \times_M B \times_M C$ as a double vector bundle and inverting the relation between $\mathcal{D}$ and $\rmd_E$ given above. Thus, up to these non-canonical choices, $\mathcal{VB}$-algebroids and two term representation up to homotopy are equivalent.\par
This result is of course not new. Gracia-Saz \& Mehta \cite{Gracia-Saz:2009} proved that isomorphisms classes of $\mathcal{VB}$-algebroids are in one-to-one correspondence with gauge equivalence classes of   two term representation up to homotopy.  The r\^ole of Lie algebroid modules as defined by Va\u{\i}ntrob \cite{Vaintrob:1997} in the super-language, was made very explicit by   Mehta \cite{Mehta:2014}. \par
From our perspective, the canonical $A$-module $\big( \Omega^{(1, \bullet)}(E), \rmd_E \big)$ associated with a $\mathcal{VB}$-algebroid is truly the fundamental generalisation of a representation of a Lie algebroid. Understanding the  $\mathcal{VB}$-algebroid as the `geometric model'  of a `higher representation' in this picture  requires no non-canonical choices.
\begin{example}[The adjoint module]
The $\mathcal{VB}$-algebroid in question here is the cotangent prolongation of a Lie algebroid $A$. As a double vector bundle we have $\sT^* A \simeq \sT^* A^*$. In homogeneous coordinates
$$\big( \underbrace{x^a}_{(0,0)} ,~ \underbrace{y^i}_{(0,1)} , ~ \underbrace{z_j}_{(1,0)},~ \underbrace{p_b}_{(1,1)}   \big),$$
elements of $\Omega^{(1,\bullet)}(\sT^*A)$ are of the form
$$\sigma =  \sum_{n=0}^q   \frac{1}{n!} y^{i_1} \wedge \dots \wedge y^{i_n} \sigma^j_{i_n \dots i_1}(x) z_j  +  \sum_{n=0}^q   \frac{1}{n!} y^{i_1} \wedge \dots \wedge y^{i_n} \sigma^a_{i_n \dots i_1}(x) p_a. $$
The differential associated with the cotangent prolongation of $A$ is given by
\begin{align*}
& \rmd_{\sT^*A} z_i = Q_i^a p_a + y^jQ_{ji}^k z_k,
&& \rmd_{\sT^*A} p_a = - y^i \frac{\partial Q_i^b}{\partial x^a} p_b - \frac{1}{2}y^i \wedge y^j \frac{\partial Q_{ji}^k}{\partial x^a}z_k.
\end{align*}
Via a non-canonical splitting we can make the identification $\mathcal{V} = A \oplus \sT M$ , where $A$ sits in degree $0$, and $\sT M$ sits in degree $1$,  and we recover the adjoint representation (cf. \cite{Abad:2012}).
\end{example}

\section{Weighted Lie algebroids and representations up to homotopy}\label{sec:WeightedAlgebroids}
\subsection{Weighted Lie algebroids}
 One can think of  a \emph{weighted Lie algebroid} as a Lie algebroid in the category of graded bundles or, equivalently, as a graded bundle in the category of Lie algebroids, see~\cite{Bruce:2014b,Bruce:2015,Bruce:2016}. Thus one should think of weighted Lie algebroids as Lie algebroids that carry an additional compatible grading. While there are several equivalent ways to define a weighted Lie algebroid, the most direct and suitable for our purposes is the following.
\begin{Definition}
A \emph{weighted Lie algebroid} of degree $k$ is a Lie algebroid $\big(E \to B, \; [-,-]_E , \; \rho_E  \big)$   equipped with a homogeneity structure of degree $k$ such that
$$ \rmh:  \,  \R \times  E \rightarrow  E,$$
acts as  a Lie algebroid morphism for all  $t \in \mathbb{R}$.
\end{Definition}
This definition of a weighted Lie algebroid is identical to the definition of a $\mathcal{VB}$-algebroid, except for the fact that we no longer insist that the homogeneity structure be regular. Note that a weighted Lie algebroid of degree one is precisely a $\mathcal{VB}$-algebroid.  For brevity, we will often denote a weighted Lie algebroid simply by $E$.
\begin{remark}
 In fact, dropping the regularity not only widens the class of objects we can consider, but also drastically simplifies working with them. Specifically, one does not need to check that some technical condition akin to regularity is preserved in the constructions. However, one may need to check that the degree of the homogeneity structure is preserved, but this is a local requirement that can usually be readily checked.   For example, the Lie theory relating weighed Lie algebroids and weighted Lie groupoids is technically simpler than just considering $\mathcal{VB}$-objects (see \cite{Bruce:2015} for details).
\end{remark}

Like $\mathcal{VB}$-algebroids, every weighted Lie algebroid is a double graded bundle, i.e.~a manifold equipped with two commuting homogeneity structures: one is $\rmh$, the other one is the regular homogeneity structure corresponding to the vector bundle $E \to B$.  In particular,  $E$ can be coordinatised by homogeneous coordinates of bi-weights $(i, j)$ where $0 \leq i \leq k$ and $0\leq j \leq 1$. That is we will order the bi-weight as the weight associated with the homogeneity structure $\rmh$ first and then the weight associated with the vector bundle structure. Let us define $ E_i :=  E_{(i,1)}$,  where $E_{(i,1)}$ is the reduction of $E$ to coordinates of degree $\le i$ with respect to the first component of the bi-weight, $1\leq i \leq k$. Furthermore, for notational ease we set $E_{0} =:  A$ and $E_{(i,0)} =:B_i$ and $B_0 =: M$. From the structure of a weighted Lie algebroid of degree $k$  we have the following tower
\begin{equation}\label{eq:graded_LA}
\xymatrix{E =  E_k \ar[r] \ar[d] & E_{k-1} \ar[r] \ar[d] & E_{k-2} \ar[r] \ar[d] & \cdots \ar[r] & E_1 \ar[r] \ar[d] & A \ar[d] \\
 B = B_k \ar[r]  & B_{k-1} \ar[r]  & B_{k-2} \ar[r]  & \cdots \ar[r] & B_1 \ar[r]  & M }
\end{equation}
which defines a series of weighted Lie algebroids each of one degree lower than the previous (see \cite{Bruce:2016} for details). Moreover, each vertical arrow defines a Lie algebroid via forgetting the weight.   In particular, the vector bundle $\pi : A \rightarrow M$ is a `standard' or `ungraded' Lie algebroid. Following our earlier nomenclature, $A$ is the \emph{underlying degree zero Lie algebroid} of the weighted Lie algebroid $E$.

\subsection{Weighted Lie algebroids, Lie algebroid modules and representations}
Given the relation between $\mathcal{VB}$-algebroids, Lie algebroid modules and two-term representations up to homotopy, it is natural to wonder what survives passing to weighted Lie algebroids.  In particular, what are  the analogues of the canonical $A$-modules associated with  $\mathcal{VB}$-algebroids and what are their relation to representations up to homotopy?\par
As a Lie algebroid $E$ has a perfectly well defined Chevalley--Eilenberg--de Rham cochain complex. Let us pick homogeneous coordinates on $E$ of the form
$$\big(  \underbrace{X^A_u }_{(u, 0)} ,\: \underbrace{Y^I_{u}}_{(u,1)} \big),$$
where $0 \leq u \leq k$. Lie algebroid forms are constructed locally by taking wedge products of $Y$'s and their coefficients are smooth functions in $X$. Thus, the notion of \emph{homogeneous $E$-forms} makes perfect sense. As for $\mathcal{VB}$-algebroids,  the action of the multiplicative reals on $E$-forms is just via pullback.
\begin{Definition}
A \emph{homogeneous $E$-form} of weight $i \in \mathbb N$ is an $E$-form $\omega \in \Omega^\bullet(E)$ such that  $\rmh^*_t \omega = t^i \: \omega.$
\end{Definition}
We denote by $\Omega^{(i,j)}(E)$ the space of homogeneous $j$-forms of weight $i$. The space of all homogeneous $E$-forms is  $\Omega^{(\bullet, \bullet)}(E) := \oplus_i \Omega^{(i, \bullet)}(E)$. Note that we in fact have a natural module structure over $\Omega^{(0, \bullet)}(E) \cong \Omega^\bullet(A)$. Furthermore, the differential $\rmd_E$ is of bi-weight $(0,1)$ and thus  $\Omega^{(i, \bullet)} (E)$ is a subcomplex in $(\Omega^\bullet (E), \rmd_E)$ for all $i$. \par
For completeness, the differential $\rmd_E$ can be locally expressed via
\begin{align*}
& \rmd_E X_u^A  = Y^I_{u-v}Q_I^A[v](X),\\
& \rmd_E Y^k_{u} = \frac{1}{2} Y^I_{w-v} \wedge Y^J_{u-w}Q^K_{JI}[v](X),
\end{align*}
where $Q_I^A[v] $ and $Q^K_{JI}[v]$ are the homogeneous parts of the structure functions of weight $v$  $(0 \leq v \leq k)$. Here we understand any $Y$ that seemingly negative weight as being set to zero. Moreover, it is not hard to see that reducing to $u=0$ gives the differential $\rmd_A$ associated with the underlying degree zero Lie algebroid $A$.\par
\begin{theorem}\label{thm:EModules}
Let $E \to B$ be a weighted Lie algebroid of degree $k$ and let $A \to M$ be its underlying degree zero Lie algebroid, as in (\ref{eq:graded_LA}). Then  each $\big( \Omega^{(i, \bullet)}(E), \rmd_E \big) $  with $1 \leq i \leq k$ is an $A$-module.
\end{theorem}
\begin{proof}
Consulting Definition \ref{def:LieModule}, it is clear from the constructions that part 1. and part 2. are satisfied. Only part 3., which is essentially a `local freeness' condition, needs careful examination. First, we use Proposition \ref{prop:split} to present $B$ in the form
\[
B \cong \bar\sv(B_k) \oplus_M \bar\sv(B_{k-1}) \oplus_M \dots \oplus_M \bar\sv(B_1) \,,
\]
and $E$ in the form
\[
E \cong \bar\sv(E_k) \oplus_A \bar\sv(E_{k-1}) \oplus_A \dots \oplus_A \bar\sv(E_1) \,.
\]
Again we stress that these isomorphisms are in general non-canonical. However, it is easy to see that $\bar\sv (E_l)$ is a vector bundle over $\bar\sv (B_l)$   $( 1 \leq l \leq k)$  in a canonical way. Even more,  
\begin{equation}\label{DVB}
\xymatrix{\bar\sv (E_l) \ar[r] \ar[d] & A \ar[d] \\
 \bar\sv (B_l)  \ar[r] & M}
\end{equation}
is a double vector bundle, where we have a shift in the weight of the fibre coordinates in accordance with Proposition \ref{prop:split}. This follows  directly from the permissable changes of local coordinates on  $E$.  Hence 
\[
\bar\sv (E_l) \cong A \oplus_M \bar\sv (B_l) \oplus_M C_l
\]
for some vector bundle $C_l \to M$ (which is of course the core of (\ref{DVB})). It follows that
\[
\Omega^\bullet (E) \cong C^\infty (B) \otimes \Omega^\bullet (A) \otimes \Omega^\bullet (C_1 \oplus \cdots \oplus C_k),
\]
where the tensor products are over $C^\infty (M)$. Hence
\begin{equation}\label{cong}
\Omega^{(i, \bullet)} (E) \cong \Omega^\bullet (A) \otimes  \Sec (\mathcal V^{(i)})
\end{equation}
where 

\[
\mathcal V^{(i)} = \sum_\star \rmS^{l_1} \bar \sv (B_1)^\ast \otimes \cdots \otimes \rmS^{l_k} \bar \sv (B_k)^\ast \otimes \wedge^{m_1} C_1^\ast \otimes \cdots \otimes \wedge^{m_k} C_k^\ast.
\]
and the sum $\sum_\star$ is over all $l_1, \ldots, l_k, m_1, \ldots, m_k$ such that
\[
l_1 + m_1 + 2(l_2 + m_2) + \cdots + k (l_k + m_k) = i .
\]
Notice that, if we understand $\mathcal V^{(i)}$ as a graded vector bundle by putting 
\[
\mathcal V^{(i)} = \bigoplus_j \mathcal V^{(i, j)},
\]
for
\[
\mathcal V^{(i, j)} = \sum_{\star, \, m_1 + \cdots + m_k = j} \rmS^{l_1} \bar \sv (B_1)^\ast \otimes \cdots \otimes \rmS^{l_k} \bar \sv (B_k)^\ast \otimes \wedge^{m_1} C_1^\ast \otimes \cdots \otimes \wedge^{m_k} C_k^\ast,
\]
then isomorphism (\ref{cong}) is an isomorphism of graded modules. This concludes the proof. 
\end{proof}

\begin{remark}\label{rem:LUCA}
Any element in $\Omega^{(i, \bullet)}(E)$ is locally given by at most a (wedge) product of $i$ coordinates of bi-weight not equal to $(0,0)$ or $(0,1)$ and then an element of $\Omega^\bullet(A)\: (\cong \Omega^{(0,\bullet)}(E))$. Note that any element from $\Omega^{(i,j)}(E)$ with $j >i $ must contain a factor from  $\Omega^\bullet(A)$. Now consider the ideal $\mathcal{J}(E)$ in $\Omega^{(\bullet, \bullet)}(E)$ generated by $\Omega^\bullet(A)$. We denote the $\Omega^{(i,j)}(E)$ component of this ideal as $\mathcal{J}^{(i,j)}(E)$.  We then define
$$W^{(i,j)}(E) := \Omega^{(i,j)}(E)\, / \, \mathcal{J}^{(i,j)}(E),$$
the quotient $C^\infty (M)$-module.  Note that $W^{(i,j)}(E)$ vanishes if $j>i$. Then considering all $j \in \mathbb{N}$ we have
$$W^i(E) :=  \bigoplus_jW^{(i,j)}(E) = \bigoplus_{j\leq i}W^{(i,j)}(E). $$
In this way we obtain a finite direct sum of  $C^\infty(M)$-modules. The proof of Theorem \ref{thm:EModules} now shows that $W^{(i,j)} (E) \cong \Sec (\mathcal V^{(i,j)})$, so that $W^i (E) \cong  \Sec (\mathcal V^{(i)})$, and $\Omega^{(i, \bullet)} (E) \cong \Omega^\bullet (A) \otimes W^i (E)$ where the tensor product is over $C^\infty (M)$.
\end{remark}

\begin{example}\label{exp:E3}
 To illustrate the above theorem, consider a weighted Lie algebroid of degree $2$. We assume once for all that we fixed isomorphisms $W^{(i,j)} (E) \cong \Sec (\mathcal V^{(i,j)})$ and $\Omega^{(i, \bullet)} (E) \cong \Omega^\bullet (A) \otimes W^i (E)$ as in Remark  (\ref{rem:LUCA}). The Lie algebroid structure is somewhat superfluous for this illustration as it is the graded structure that is important for constructing the $C^\infty(M)$-modules.  The two modules that we need to  construct are
 \begin{align*}
 & W^1(E) = W^{(1,0)}(E) \oplus W^{(1,1)}(E),\\
 & W^2(E) = W^{(2,0)}(E) \oplus W^{(2,1)}(E) \oplus W^{(2,2)}(E),
 \end{align*}
 as all higher weight modules are trivial. Let us choose some local coordinates on $E_3$
$$\big( \underbrace{x^a}_{(0,0)}, \:  \underbrace{z^\alpha}_{(1,0)},\: \underbrace{u^\delta}_{(2,0)} \: ; \:  \underbrace{y^i}_{(0,1)},\: \underbrace{w^l}_{(1,1)}; \: \underbrace{v^p}_{(2,1)}\big). $$

In other words, the $x$ are coordinates on $M$, the $z, u, y, w, v$ are linear fiber coordinates on $\bar \sv (B_1), \bar \sv (B_2), A, C_1, C_2$ respectively. We can then construct a local basis using these coordinates. Typical elements are then (locally) of the form
\begin{align*}
\sigma[1] &= \:  z^\alpha \sigma(x, y) + w^l\sigma_l(x, y) & \in W^1(E),\\
\sigma[2] &=  \:  u^\delta \sigma_\delta(x,y) + \frac{1}{2}z^\alpha z^\beta  \sigma_{\beta \alpha}(x,y) \\ & \quad +  \: v^p\sigma_p(x,y) + z^\gamma w^l \sigma_{l \gamma}(x,y) + \frac{1}{2} w^m \wedge w^n \sigma_{nm}(x,y) & \in   W^2(E).
\end{align*}

Next we should consider what happens under changes of coordinates. Symbolically we write
\begin{align*}
& z \mapsto z, && u \mapsto u + z^2,\\
& y \mapsto y, && w \mapsto w + yz,\\
& v \mapsto v + zw + yu + y z^2,
\end{align*}
where we have neglected to write indices, numerical factors,  wedge products  and dependency on $x^a$. The above are the general form of permissible changes of homogeneous coordinates on $E_2$. Note that disregarding the linear dependence on $y$ of these coordinate changes, the (local) basis elements for each homogeneous component of $W^1(E)$ and $W^2(E)$ transforms as linear combinations of each other. Thus, the local constructions of $W^1(E)$ and $W^2(E)$ is well defined and can then be globalised via gluings.
\end{example}
Note that the construction of the $A$-modules $\Omega^{(i,\bullet)}(E)$ is canonical and requires no choices to be made. The non-canonical aspect is in finding an isomorphism as $C^\infty(M)$-modules to modules of the form $\Omega^\bullet(A,\mathcal{V})$, and this can be done via a choice of homogeneous coordinates on $E$. However, the definition of an $A$-module requires only that such isomorphisms exist, not that that they be fixed as part of the structure. Once an isomorphism is chosen we obtain a representation up to homotopy of $A$. Thus a direct consequence of Theorem \ref{thm:EModules} is the following.
\begin{corollary}
Let $E$ be a weighted Lie algebroid of degree $k\geq 1$, and let $A$ be the underlying degree zero Lie algebroid. Then associated with the canonical $A$-modules $\big(\Omega^{(i,\bullet)}, \rmd_E\big)$ $(1\leq i \leq k)$ is (up to isomorphism classes) a series of representations up to homotopy of $A$  with $2, 3, \cdots, k+1$ terms, respectively.
\end{corollary}
It must be noted that the Lie algebroid modules, and so the representations up to homotopy, that we obtain are very particular. The Lie algebroid modules consist of certain polynomials in homogeneous functions on a weighted Lie algebroid. Examining Example \ref{exp:E3}, the reader should note that $W^2(E)$ contains terms involving $z^\gamma w^l$, $z^\alpha z^\beta$ and  $w^m \wedge w^l$. Moreover, in general such terms are necessary due to the permissible changes of local coordinates on $E_2$. Written somewhat heuristically, ``$W^1(E)\wedge W^1(E) \subset W^2(E)$''. Similar statements hold when considering weighted Lie algebroids of higher degree.   Thus, there is no way that our constructions can be used to obtain a general $k+1$-term representation up to homotopy of a given Lie algebroid $A$. This is quite different to the very special case of $\mathcal{VB}$-algebroids and $2$-term representations up to homotopy. Essentially,  the fact that the Lie algebroid modules obtained from  $\mathcal{VB}$-algebroids are linear in certain homogeneous coordinates means that one can describe all $2$-term representations up to homotopy geometrically in this way.

\subsection{Some simple examples}
We now turn our attention to some illustrative examples. Given that weighted Lie algebroids are plentiful, see \cite{Bruce:2014b,Bruce:2016}, many further intricate examples can be constructed. 

\begin{example}
In \cite{Bruce:2016} the notion of a weighted Lie algebra was defined as a weighted Lie algebroid for which $M = \{ \textnormal{pt}\}$. This means that the underlying degree zero Lie algebroid is in fact a Lie algebra, we denote this vector space as $\mathfrak{g}$. However, we still have a `double structure'
\begin{equation*}
\xymatrix{E = E_k \ar[r] \ar[d] & \mathfrak{g} \ar[d] \\
 B = B_k  \ar[r] & \{\textnormal{pt} \}}
\end{equation*}
We employ homogeneous coordinates
$$\big( \underbrace{z^\alpha_w}_{(w,0)}, \underbrace{y^i}_{(0,1)} , \underbrace{v_w^p}_{(w,1)}\big),$$
 $(1\leq w \leq k)$ on $E_k$.  The corresponding differential is  given by
 \begin{align*}
 & \rmd_E z^\alpha_w = y^i Q_i^\alpha[w](z)+ v^p_{w-u}Q_p^\alpha[u](z),\\
 & \rmd_E y^i = \frac{1}{2}y^l \wedge y^k Q_{kl}^i,\\
 & \rmd_E v^p_w = \frac{1}{2}y^l \wedge y^k Q_{kl}^p[w](z) + y^i \wedge v^q_u Q_{qi}^p[w-u](z) + \frac{1}{2} v^q_{u-v} \wedge v^r_{w-u}Q_{rq}^p
[v](z), 
\end{align*}
where we have employed our earlier convention of setting terms of weight outside there specified range to vanish. Elements of  $\Omega^{(i,\bullet)}(E)$, where   $1\leq i \leq k$, are of the local form
$$ \sigma[i] = \sum_{\substack{w_{1}+ \cdots + w_{n} \\ + v_{1}+ \cdots + v_{m} = i}} \frac{1}{n! m!}v_{w_{1}}^{p_{1}} \wedge \cdots \wedge v_{w_{n}}^{p_{n}}  \: z_{v_{1}}^{\alpha_{1}}\cdots z_{v_{m}}^{\alpha_{m}} \sigma_{\alpha_{m}\cdots \alpha_{1}   p_{n} \cdots p_{1 }} (y),$$
where we allow $n$ and $m$ (but not both) to be zero. We have the non-canonical isomorphism
$$\Omega^{(\bullet, \bullet)}(E) \cong \Omega^\bullet(\mathfrak{g}) \otimes \A(B) \otimes \Omega^\bullet(C_1 \oplus C_2 \dots \oplus C_k), $$
where $\A(B)$ is the polynomial algebra on the graded vector space $B$ and $C_l$ are the core vector spaces  of 
\begin{equation*}
\xymatrix{\bar\sv (E_l) \ar[r] \ar[d] & \mathfrak{g} \ar[d] \\
 \bar\sv (B_l)  \ar[r] & \{\textnormal{pt} \}}
\end{equation*}
Thus we identify each vector space $W^i(E)$ with the space of  anti-symmetric forms on the direct sum of the cores that take their values in the polynomial algebra on the graded vector space $B$ that are of weight $i$.
\end{example}

\begin{example}\label{E7}
Continuing Example \ref{exp:E3}, in homogeneous local coordinates the differential $\rmd_E$, and so the flat superconnection associated with the representations up to homotopy  of $A$, is defined by the standard differential $\rmd_A$ and
\begin{align*}
\rmd_E z^\alpha = {}& z^\beta y^i Q_{i \beta}^\alpha(x) + w^l Q_l^\alpha(x),\\
\rmd_E u^\delta  = {}&u^\epsilon y^i  Q_{i \epsilon }^\delta(x) + \frac{1}{2} z^\alpha z^\beta y^i Q_{i \beta \alpha }^\delta(x) + z^\alpha w^l Q_{l \alpha}^\delta(x) + v^pQ_p^\delta(x),\\
 \rmd_E w^l  = {}&y^i \wedge w^m Q_{mi}^l(x) + \frac{1}{2}y^i \wedge y^j z^\alpha Q_{\alpha ji}^l(x),\\
 \rmd_E v^p = {}&y^i \wedge v^q Q_{qi}^p(x) + \frac{1}{2} y^i \wedge y^j u^\delta Q_{\delta ji}^p(x)  + \frac{1}{4} y^i \wedge y^j
 z^\alpha z^\beta Q_{\beta \alpha ji}^p(x)\\
  &+ y^i \wedge w^l z^\alpha Q_{\alpha l i}^p(x) + \frac{1}{2}w^l \wedge w^m Q_{ml}^p(x).
 \end{align*}
 Thus, up to isomorphism classes we have associated with $E_2$
 \begin{enumerate}
 \item a \emph{two} term representation up to homotopy of $A$;
 \item a \emph{three} term representation up to homotopy of $A$.
 \end{enumerate}
\end{example}
\begin{example}[Tangent bundles of graded bundles]\label{exp:TGradedBundle} Consider a graded bundle $\tau: F_{k} \longrightarrow M$ of degree $k$. The tangent bundle  $\sT F_{k}$ is naturally a weighted Lie algebroid of degree $k$. We can specify this weighted Lie algebroid structure via the canonical de Rham differential on $\Omega^\bullet(F_k)$ (see \cite{Bruce:2016} for details). In natural local coordinates
$$(\underbrace{x^{a}}_{(0,0)}, \:  \underbrace{z_{w}^{\alpha}}_{(w, 0)},\: \underbrace{\rmd x^{b}}_{(0,1)} , \:\underbrace{\rmd z^\beta_{w}}_{(w,1)}),$$
where $1\leq w \leq k$ the canonical de Rham can be expressed as
\begin{align*}
&\rmd_{\sT F}x^a = \rmd x^a, && \rmd_{\sT F} z^\alpha_w = \rmd z^\alpha_{w}.
\end{align*}
Clearly we have that the underlying degree zero Lie algebroid is the tangent bundle of $M$ equipped with its canonical Lie algebroid structure.  Elements of  $\Omega^{(i,\bullet)}(\sT F)$, where   $1\leq i \leq k$, are of the local form
$$ \sigma[i] = \sum_{\substack{w_{1}+ \cdots + w_{n} \\ + v_{1}+ \cdots + v_{m} = i}} \frac{1}{n! m!}\rmd z_{w_{1}}^{\alpha_{1}} \wedge \cdots \wedge \rmd z_{w_{n}}^{\alpha_{n}}  \: z_{v_{1}}^{\beta_{1}}\cdots z_{v_{m}}^{\beta_{m}} \sigma_{\beta_{m}\cdots \beta_{1}   \alpha_{n} \cdots \alpha_{1 }} (x, \rmd x),$$
where we allow $n$ and $m$ (but not both) to be zero. Thus, we can identify the $C^{\infty}(M)$-modules $ W^i(\sT F)$ as the modules whose elements are the sums of vertical differential forms  on $ \tau: F_{k} \rightarrow M$ that are of weight $i$.
\end{example}

\begin{example}[Prolongations of graded bundles]
Given a Lie algebroid $A$ and a graded bundle $F_k$ both over the same base manifold $M$ one can construct the weighted Lie algebroid of degree $k$ given by
$$ E = \mathcal{T}^{A}F_{k} :=  \left \{ (\mathsf{u}, \mathsf{v}) \in A \times_M \sT F_{k} ~ | ~  \rho(\mathsf u )  = \sT \tau_{k}(\mathsf v)\right\}.$$
Lie algebroids of this kind are known in the literature as \emph{pull-back Lie algebroids} \cite{Mackenzie2005} or \emph{prolongations of fibred manifolds with respect to a Lie algebroid}, see \cite{Bruce:2016} and the references given therein. This weighted Lie algebroid structure is  easiest to describe via the differential $\rmd_E$. With this in mind, natural homogeneous coordinates on $E$ are of the form
$$\big(\underbrace{x^a}_{(0,0)},\: \underbrace{z^\alpha_w}_{(w,0)},\: \underbrace{y^j}_{(0,1)}, \: \underbrace{\rmd z^\beta_{w}}_{(w,1)}    \big),$$
where $1\leq w \leq k$. Then we can express $\rmd_E$ locally as
\begin{align*}
& \rmd_E x^a = y^i Q_i^a(x), && \rmd_E z^\alpha_w = \rmd z^\alpha_{w}, & \rmd_E y^k = \frac{1}{2} y^i \wedge y^j Q_{ji}^k(x).
\end{align*}
Elements of  $\Omega^{(i,\bullet)}(E)$, where   $1\leq i \leq k$, are of the local form
$$ \sigma[i] = \sum_{\substack{w_{1}+ \cdots + w_{n} \\ + v_{1}+ \cdots + v_{m} = i}} \frac{1}{n! m!}\rmd z_{w_{1}}^{\alpha_{1}} \wedge \cdots \wedge \rmd z_{w_{n}}^{\alpha_{n}}  \: z_{v_{1}}^{\beta_{1}}\cdots z_{v_{m}}^{\beta_{m}} \sigma_{\beta_{m}\cdots \beta_{1}   \alpha_{n} \cdots \alpha_{1 }} (x,  y),$$
where we allow $n$ and $m$ (but not both) to be zero. Here we understand the components $\sigma_{\beta_{m}\cdots \beta_{1}   \alpha_{n} \cdots \alpha_{1 }} (x,  y)$ to be (collections of) elements from $\Omega^\bullet(A)$. Thus, we can identify the $C^{\infty}(M)$-modules $ W^i(E)$ as the modules whose elements are the sums of vertical differential forms  on $ \tau: A \rightarrow M$ that are of weight $i$.
This is almost exactly the same situation as Example \ref{exp:TGradedBundle}, but now the underlying weight zero Lie algebroid is $A$ and not the tangent bundle $\sT M$.
\end{example}
\subsection{A supergeometric interpretation}
 For readers familiar with the language of supergeometry we provide a supergeometric interpretation of our observations. First of all, the Va\u{\i}ntrob's interpretation of a Lie algebroid on a vector bundle $A$ is a homological vector field $Q$ on $\Pi A$ (where $\Pi$ is the fibre parity reversing functor) of degree 1, $Q^2=0$. The Grassmann algebra $\Omega^\bullet(A)$ is interpreted as the algebra of smooth functions on $\Pi A$ and $Q=\rmd_A$. For instance, the Lie algebroid structure (\ref{Q}) corresponds to the homological vector field 
$$\rmd_{A} = \zx^{i}Q_{i}^{a}(x) \frac{\partial}{\partial x^{a}} + \frac{1}{2} \zx^{i}\zx^{j}Q_{ji}^{k}(x)\frac{\partial}{\partial \zx^{k}}\,,$$
where $(x^a,\zx^i)$ are coordinates on $\Pi A$ associated to a basis $(s_i)$ of local sections.\par 
$A$-modules are represented by graded (super)vector bundles $\cE$ over $\Pi A$ equipped with a homological vector field of degree 1 which projects onto $\rmd_A$. In the supergeometric interpretation, a weighted Lie algebroid is a homological vector field $Q$ on a superized vector bundle $\Pi E$ equipped with a compatible $\N$-grading, which is of bi-degree $(0,1)$. For a double vector bundle (\ref{dvb}) we consider the vector bundle $\Pi_BE\to \Pi A$, now in the category of supermanifolds, with homogeneous (super)coordinates:
$$(\underbrace{x^{a}}_{(0,0)}, \underbrace{z^{\alpha}}_{(1,0)}, \underbrace{\zx^{i}}_{(0,1)} , \underbrace{\theta^{l}}_{(1,1)}).$$
Note that the second component of the bi-weight corresponds to the Grassmann parity of the coordinates. Then in these coordinates the  homological vector field associated with the $\mathcal{VB}$-structure is given locally by
\begin{eqnarray}\label{eqn:Q}
Q &=& \zx^{i}Q_{i}^{a}\frac{\partial}{\partial x^{a}} + \frac{1}{2}\zx^{i}\zx^{j}Q_{ji}^{k}\frac{\partial}{\partial \zx^{k}} + \left(z^\beta \zx^i Q_{i \beta}^\alpha + \theta^l Q_l^\alpha \right)\frac{\partial}{\partial z^{\beta}} \\
\nonumber &+& \left ( \zx ^i \theta^m Q_{mi}^l + \frac{1}{2} \zx^i  \zx^j z^\alpha Q_{\alpha ji}^l \right)\frac{\partial}{\partial \theta^{l}},
\end{eqnarray}
where the coefficients are functions of $x$. Sections of $\Pi_{B}E$ locally have the form
$$\sigma = z^{\alpha}\sigma_{\alpha}(x, \zx) + \theta^{l}\sigma_{l}(x,\zx).$$
Thus we identify the graded bundle in the definition of a two-term representation up to homotopy with $\mathcal{V} \simeq B^{*} \oplus \Pi C^{*}$, where $C$ is the core of $E$. Hence, up to isomorphism classes, we have a one-to-one correspondence between $\mathcal{VB}$-algebroids and two term representations up to homotopy.\par 
This procedure then extends to the more general setting directly. In particular, on $\Pi E$ we employ local coordinates 
$$\big( \underbrace{X_u^A}_{(0,u)  }, \:  \underbrace{\Theta_u^I}_{(u,1)} \big),$$
 $(0 \leq u \leq k)$ where the second component of the bi-weight again corresponds to the Grassmann  parity of the coordinates. In all local expressions for the $A$-modules and differentials one then replaces wedge products of $Y$'s with supercommutative products of $\Theta$'s. For example, the Chevalley--Eilenberg--de Rham differential is then understood as a homological vector field on $\Pi E$ given locally as
$$Q = \Theta^{I}_{u-u'+1}Q_{I}^{\alpha}[u'](X) \frac{\partial}{\partial X^{\alpha}_{u} }  + \frac{1}{2} \Theta^{I}_{u''-u'+1} \Theta^{J}_{u-u''+1}Q_{JI}^{K}[u'](X)\frac{\partial}{\partial \Theta_{u+1}^{K} } ,$$
 where $Q_{I}^{\alpha}[u']$ and $Q_{IJ}^{K}[u']$ are the homogeneous parts of the structure functions of degree~$u'$. In the notation employed here, any  $\Theta$ with seemingly zero or negative total weight  is set to zero. All the constructions and examples that have so far appeared in this paper can be re-formulated within this supergeometric framework. 

\section{The weighted van Est theorem}\label{sec:VanEst}
\subsection{The van Est map}

Let $\mathcal G \rightrightarrows M$ be a Lie groupoid and let $A := A(\mathcal G) \to M$ be its Lie algebroid. We denote by $\mathcal G^{(p)}$ the manifold of $p$ composable arrows. Let $(C^\bullet (\mathcal G), \delta)$ be the complex of smooth (normalized) groupoid cochains of $\mathcal G$, i.e.~smooth maps $c : \mathcal G^{(\bullet)} \to \mathbb R$, such that $c (g_1, \ldots, g_p) = 0$ whenever $g_i = 1$ for some $i$. Recall that there is a \emph{van Est} differential graded algebra map
\[
\mathsf{VE}: (C^\bullet (\mathcal G), \delta) \to (\Omega^\bullet (A), \rmd_A)
\]
defined as follows. For a cochain $c \in C^p (\Gamma)$, and all sections $\alpha_1, \ldots, \alpha_p$ of $A \to M$:
\[
\mathsf{VE} (c) (\alpha_1, \ldots, \alpha_p) = \sum_{\sigma \in S_p} (-)^\sigma R_{\alpha_{\sigma(1)}} \cdots R_{\alpha_{\sigma (p)}} (c).
\]
Here
\[
R_\alpha : C^\bullet (\mathcal G) \to C^{\bullet - 1} (\mathcal G)
\]
is the map defined by:
\[
(R_\alpha c)(g_1, \ldots, g_{p-1}) = \frac{\rmd}{\rmd\epsilon}|_{\epsilon = 0} c(\phi_\epsilon^\alpha \circ \mathsf s (g_1), g_1, \ldots, g_{p-1})
\]
where $\{ \phi_\epsilon^\alpha \}$ is the flow of the right invariant vector field $\overrightarrow \alpha$ on $\mathcal G$ corresponding to section $\alpha$ of $A \to M$, and $\mathsf s : \mathcal G \to M$ is the source map.\par 
It follows that there is a graded algebra map in cohomology, also denoted $\mathsf{VE}$:
\begin{equation}\label{eq:VE}
\mathsf{VE} : H^\bullet (C(\mathcal G), \delta) \to H^\bullet (\Omega^\bullet(A), \rmd_A).
\end{equation}
The \emph{van Est theorem} of Crainic \cite{Crainic:2003}, asserts that, if the source fibres of $\mathcal G$ are $p_0$-connected, then (\ref{eq:VE}) is an isomorphism in degree $p \leq p_0$ and it is injective in degree $p_0+1$. For more details on groupoid cohomology and the van Est map the reader should consult \cite{Crainic:2003,Li-Bland:2015} and reference therein.

\subsection{Extending to the weighted category}
Recently, Cabrera \& Drummond \cite{Cabrera:2016} proved a refinement of van Est theorem for \emph{homogeneous cochains} of a $\mathcal{VB}$-groupoid. Let us briefly review their result. Recall that a $\mathcal{VB}$-groupoid is a \emph{vector bundle object} in the category of Lie groupoid, i.e.~it is a diagram
\[
\begin{array}{c}
\xymatrix{ \Gamma \ar[r] \ar@<2pt>[d] \ar@<-2pt>[d]& \mathcal G \ar@<2pt>[d] \ar@<-2pt>[d] \\
B \ar[r] & M}
\end{array}
\]
also written $(\Gamma, B; \mathcal G, M)$, where
\begin{itemize}
\item the columns are Lie groupoids,
\item the rows are vector bundles,
\item all the vector bundle structure maps are Lie groupoid maps.
\end{itemize}
There is a vastly simpler definition of a $\mathcal{VB}$-groupoid, (see \cite{Bursztyn:2014}). Namely, a $\mathcal{VB}$-groupoid is the same as a Lie groupoid $\Gamma \rightrightarrows B$ equipped with a  homogeneity structure $\rmh : \mathbb R \times \mathcal G \to \mathcal G$ such that
\begin{enumerate}
\item $\rmh_t$ is a groupoid map for all $t$,
\item $(\Gamma, \rmh)$ is a degree $1$ weighted bundle (in other words, $\rmh$ is a regular homogeneity structure).
\end{enumerate}
Then $\mathcal G = \rmh_0 (\Gamma)$ and $M = \rmh_0 (B)$. If we drop the second condition, we get the far more general notion of a \emph{weighted Lie groupoid} \cite{Bruce:2015}. It can be checked that the Lie algebroid of a $\mathcal{VB}$-groupoid is a $\mathcal{VB}$-algebroid. More generally, the Lie algebroid of a weighted Lie groupoid is a weighted Lie algebroid  of the same degree, see \cite{Bruce:2015}.\par
Now, let $(\Gamma, B; \mathcal G, M)$ be a $\mathcal{VB}$-groupoid, and let $(E, B; A, M)$ be its $\mathcal{VB}$-algebroid. We denote by $\rmh$ both the (regular) homogeneity structures on $\Gamma$, and $E$. They are intertwined by the Lie functor. \par
It easily follows from the $\mathcal{VB}$-groupoid axioms, that $\Gamma^{(p)}$ is a vector (sub)bundle (of $\Gamma \times \cdots \times \Gamma$) over $\mathcal G^{(p)}$, for all $p$, hence $\rmh$ induces an obvious action $\rmh^\ast : \mathbb R \times C^\bullet (\Gamma) \to C^\bullet (\Gamma)$ of multiplicative reals on groupoid cochains $C^\bullet (\Gamma)$  (via pull-back).  According to Cabrera and Drummond, a cochain $c \in C^\bullet (\Gamma)$ is \emph{homogeneous of weight $k \in \mathbb N_0$} if $\rmh_t^\ast (c) = t^k \cdot c$ for all $t \in \mathbb R$. For instance, homogeneous cochains of weight $0$ are actually cochains on $\mathcal G$, homogeneous cochains of weight $1$ are linear cochains, and so on. Cabrera \& Drummond proved the following  theorem
\begin{theorem}\label{theor:VE} \quad
\begin{enumerate}
\item The space $C^\bullet_{k\text{-hom}} (\Gamma)$ of homogeneous cochains of weight $k$ is a $C^\bullet (\mathcal G)$ submodule of $C^\bullet (\Gamma)$ and a subcomplex.
\item The van Est map $\mathsf{VE} : (C^\bullet (\Gamma), \delta) \to (\Omega^\bullet (E), \rmd_E)$ restricts to a differential graded module map $$\mathsf{VE}_{k\text{-hom}} : (C_{k\text{-hom}}^\bullet (\Gamma), \delta) \to (\Omega^{(k, \bullet)} (E), \rmd_E)$$ covering the differential graded algebra map $\mathsf{VE} : (C^\bullet (\mathcal G), \delta) \to (\Omega^\bullet (A), \rmd_A)$, for all $k$.
\item If the source fibres of $\mathcal G$ are $p_0$-connected, then $\mathsf{VE}_{k\text{-hom}}$ is an isomorphism in degree $p \leq p_0$, and it is injective in degree $p_0 +1$, for all $k$.
\end{enumerate}
\end{theorem}
It is natural to expect that Theorem \ref{theor:VE} can be generalized to arbitrary weighted Lie groupoids. This is indeed the case as proved below. So, let $(\Gamma \rightrightarrows B, \rmh)$ be an arbitrary weighted Lie groupoid, and let $(E \to B, \rmh)$ be its weighted Lie algebroid. Homogeneous cochains of $\Gamma$ can be defined precisely as above (see below for more details), and Theorem \ref{theor:VE} carries over to this general situation.
\begin{proof}[Proof of Theorem \ref{theor:VE} for a generic weighted Lie groupoid $\Gamma$]
Cabrera and Drummond proof extends without any problem to the present case. We report it here for completeness. First of all we have to make sense of the first statement. So, let $(\Gamma \rightrightarrows B, \rmh)$ be a weighted Lie groupoid and let $(E \to B, \rmh)$ be its weighted Lie algebroid. Put $\mathcal G = \rmh_0 (\Gamma)$ and $A = \rmh_0 (E)$. The diagonal action of multiplicative reals on $\Gamma \times \cdots \times \Gamma$ turns it into a weighted bundle over $\mathcal G \times \cdots \times \mathcal G$, of the same degree as $\Gamma$. As source and target maps are equivariant w.r.t.~the homogeneity structures, it easily follows that $\Gamma^{(p)}$ is a weighted bundle over $\mathcal G^{(p)}$, of the same degree as $\Gamma$, for all $p$. Indeed $\Gamma^{(p)}\subset \Gamma \times \cdots \times \Gamma$ is an invariant submanifold wrt homogeneity structure on $\Gamma \times \cdots \times \Gamma$, so that the degree of $\Gamma^{(p)}$ does not exceed that of $\Gamma$. On another hand, $\Gamma^{(p)}$ is a multiple fibered product of $\Gamma$ along submersions, so the degree of $\Gamma$ cannot exceed that of $\Gamma^{(p)}$ (we thank the anonymous referee for this remark). In particular, there is an action $\rmh^\ast$ of multiplicative reals on groupoid cochains $C^\bullet (\Gamma)$:
\[
\rmh^\ast : \mathbb R \times C^\bullet (\Gamma) \to C^\bullet (\Gamma).
\]
Namely, for $t \in \mathbb R$ and $c \in C^p (\Gamma)$, $\rmh^\ast_t c : \Gamma^{(p)} \to \mathbb R$ is defined by
\[
\rmh^\ast_t c (\gamma_1, \ldots, \gamma_p) := c (\rmh_t \gamma_1, \ldots, \rmh_t \gamma_p).
\]
A cochain $c \in C^\bullet (\Gamma)$ is \emph{homogeneous of weight k} if $\rmh^\ast_t c = t^k \cdot c$ for all $t \in \mathbb R$, and, as in the statement, we denote by $C^\bullet_{k\text{-hom}} (\Gamma)$ the space of homogeneous cochains of weight $k$. It is clear that the product of a weight $k$ and a weight $l$ cochain is a weight $k+l$ cochain. As weight $0$ cochains identify with smooth groupoid cochains on $\mathcal G$, it follows that $C^\bullet_{k\text{-hom}} (\Gamma)$ is a $C^\bullet (\mathcal G)$-module in a natural way. Additionally, from the explicit formula for the differential $\delta : C^\bullet (\Gamma ) \to C^\bullet (\Gamma)$ and the fact that $\rmh_t$ is a groupoid map for all $t$, it immediately follows that $\rmh_t^\ast : (C^\bullet (\Gamma ), \delta) \to (C^\bullet (\Gamma), \delta)$ is a cochain map. In particular, $C^\bullet_{k\text{-hom}} (\Gamma)$ is a subcomplex in $(C^\bullet (\Gamma), \delta)$ and a differential graded submodule over the differential graded algebra $(C^\bullet (\mathcal G), \delta)$. This proves the first item in the statement. \par
For the second item, it is enough to show that the van Est map $\mathsf{VE} : (C^\bullet (\Gamma), \delta) \to (\Omega^\bullet (E), \rmd_E)$ \emph{preserves the homogeneity structures}, i.e.~it is equivariant w.r.t. to the actions of multiplicative reals on $C^\bullet (\Gamma)$ and $\Omega^\bullet (E)$:
\begin{equation}\label{eq:VEw}
\mathsf{VE} (\rmh_t^\ast c) = \rmh_t^\ast \mathsf{VE}(c)
\end{equation}
for all $t$ and all groupoid cochains $c$. We will prove (\ref{eq:VEw}) for $t \neq 0$. Then, by continuity, it will also hold for $t = 0$. For an $E$-form $\omega \in \Omega^p (E)$, the pull-back $\rmh_t^\ast \omega$ looks as follows:
\[
(\rmh_t^\ast \omega)(e_1, \ldots, e_p) = \rmh_t^\ast \big( \omega (\rmh_t e_1, \ldots, \rmh_t e_p) \big), \quad t \neq 0
\]
where $\rmh_t e$ is the section of $E \to B$ defined by $\rmh_t e = \rmh_t \circ e \circ \rmh_t^{-1}$. So it is enough to check that
\[
R_{e} (\rmh_t^\ast c) = \rmh_t^\ast R_{\rmh_t e} (c).
\]
But this immediately follows from the equivariance of $\mathsf s : \Gamma \to B$, and the simple fact that $\rmh_t : \Gamma \to \Gamma$ intertwines the flows of $\overrightarrow e$ and $\overrightarrow{\rmh_t e}$. This concludes the proof of the second item.\par
For the third item, we use the same argument as Cabrera \& Drummond: for every $k$ there are projectors $P_{k\text{-hom}} : C^\bullet (\Gamma) \to C^\bullet (\Gamma)$ and $P_{k\text{-hom}} : \Omega^\bullet (E) \to \Omega^\bullet (E)$ defined by

\[
P_{k\text{-hom}} (\omega) = \frac{1}{k!} \frac{d^k}{dt^k} |_{t = 0}(h_t^\ast \omega),
\]
which makes sense for both $\omega \in C^\bullet (\Gamma)$ and $\omega \in \Omega^\bullet (E)$. The  $P_{k\text{-hom}}$ enjoy the following properties: 
\begin{enumerate}
\item the image of $P_{k\text{-hom}}$ consists of weight $k$ homogeneous cochains (resp.~$E$-forms),
\item $P_{k\text{-hom}}$ is a cochain map,
\item $P_{k\text{-hom}}$ commutes with the van Est map: $P_{k\text{-hom}} \circ \mathsf{VE} = \mathsf{VE} \circ P_{k\text{-hom}}$.
\end{enumerate}
The last item then follows from the Homological Lemma in \cite[p.~9]{Cabrera:2016}. For $P_{k\text{-hom}}$ we can use exactly the same (standard, \emph{homogenization}) formula as Cabrera and Drummon, see \cite[Equation (2.1)]{Cabrera:2016}. The properties of $P_{k\text{-hom}}$ then immediately follows from the axioms of weighted groupoids, and the fact that the Lie functor maps the homogeneity structure on $\Gamma$ to the homogeneity structure on $E$. This concludes the proof of the theorem.
\end{proof}

\begin{remark}
In the proof of Theorem \ref{theor:VE} there are several homogeneity structures playing a role: first of all those on $\Gamma$ and $E$, but also those on $\Gamma \times \cdots \times \Gamma$ and $\Gamma^{(p)}$. In the original proof of Cabrera and Drummond it is important to show that the latter are regular homogeneity structures. On the other hand, working with generic graded objects highlights that what is really important in the proof is the action of multiplicative reals on the appropriate cochain complexes and the equivariance of the appropriate (van Est) map.
\end{remark}

\section*{Acknowledgements}
The authors cordially thank the anonymous referee for their valuable comments and suggestion that have undoubtedly served to improve the presentation of this paper. LV thanks AJB and JG for their hospitality during his visit to IMPAN, Warsaw, where this project was undertaken. LV is a member of the GNSAGA of INdAM. The research JG was funded by the Polish National Science Centre grant under the contract number DEC-
2012/06/A/ST1/00256.


\end{document}